\documentclass{amsart}
\usepackage{amssymb}

\usepackage{amsfonts}

\newtheorem{theorem}{Theorem}
\theoremstyle{plain}

\newtheorem{corollary}{Corollary}

\newtheorem{lemma}{Lemma}

\newtheorem{proposition}{Proposition}
\newtheorem{remark}{Remark}

\numberwithin{equation}{section}
\input{tcilatex}

\begin{document}
\title[state-closed representations of wreath products]{On state-closed
representations of restricted wreath product of groups of type $%
G_{p,d}=C_{p}wrC^{d}$ }
\author{Alex C. Dantas}
\address{Departamento de Matem\'{a}tica, Universidade de Bras\'{\i}lia, Bras%
\'{\i}lia-DF, Brazil}
\email{alexcdan@gmail.com}
\thanks{The first author acknowledges a doctoral scholarship from CNPq.}
\author{ Said N. Sidki}
\address{Departamento de Matem\'{a}tica, Universidade de Bras\'{\i}lia, Bras%
\'{\i}lia-DF, Brazil}
\email{ssidki@gmail.com}
\thanks{The second author thanks Tatiana Smirnova-Nagnibeda\ for a visit to
Universit\'{e} de Gen\`{e}ve and thanks Dmytro Savchuk for a visit to
Southern Florida University both of which took place in 2014 and stimulated
this work.}
\date{May11th, 2015}
\subjclass[2000]{Primary20E08, 20F18}
\keywords{Tree automorphisms, state-closed representations, wreath products,
lamplighter group.}

\begin{abstract}
Let $G_{p,d}$ be the restricted wreath product $C_{p}wrC^{d}$ where $C_{p}$
is a cyclic group of order a prime $p$ and $C^{d}$ a free abelian group of
finite rank $d$. We study the existence of faithful state-closed (\textit{fsc%
}) representations of $G_{p,d}$ on the $1$-rooted $m$-ary tree for some
finite $m$. The group $G_{2,1}$, known as the lamplighter group, admits an 
\textit{fsc }representation on the binary tree. We prove that for $d\geq 2$
there are no \textit{fsc} representations of $G_{p,d}$ on the $p$-adic tree.
We characterize all \textit{fsc} representations of $\ G=G_{p,1}$ on the $p$%
-adic tree where the first level stabilizer of the image of $G$ contains its
commutator subgroup. Furthermore, for $d\geq 2$, we construct uniformly 
\textit{fsc} representations of $G_{p,d}$ on the $p^{2}$-adic tree and
exhibit concretely the representation of $G_{2,2}$ on the $4$-tree as a
finite-state automaton group.
\end{abstract}

\maketitle

\section{Introduction}

This paper is a study of representations of the restricted wreath product of
groups of type $G_{p,d}=C_{p}wrX$ where $C_{p}$ is a cyclic group of prime
order $p$ and $X$ is a free abelian group of finite rank $d\geq 1$, as
groups of automorphisms of $1$-rooted regular $m$-trees satisfying the 
\textit{state-closed property} (or, self-similarity in dynamics language);
the representations are said to be of degree $m$. \ We let $C_{p}$ be
generated by $a$, denote its the normal closure by $A$ and let $X$ be
generated by $\left\{ x_{1},x_{2},...,x_{d}\right\} $.

The groups $G_{2,d}$ appeared as examples in the study of probabilistic
properties of random walks on groups (\cite{KaimVersh}, page 480). Among
these examples, the group $G_{2,1}$ which goes by the picturesque name of
lamplighter, admits a (\textit{classical}) faithful self-similar
representation on the $2$-tree, as the state-closure of the tree
automorphism $\xi =\left( \xi ,\xi \alpha \right) $ where $\alpha $ is the
transposition automorphism. Further interest in the lamplighter group arose
from the calculation of its spectrum in \cite{GrigZuk} which was then used
to disprove a conjecture about the range of $L_{2}$-Betti numbers of closed
manifolds (see, \cite{GLSZ}). Since then, several articles have appeared on
generalizations of the lamplighter group (\cite{SilvaStein}, \cite{KSS}, 
\cite{BarthSunik}, \cite{GLN}, \cite{BonDaRo}).

An important type of representation of groups as automorphisms of trees is
that of \textit{finite-state; }or\textit{\ }equivalently, representations by 
\textit{finite automata}. It follows from a general technique called
tree-wreathing introduced in \cite{Sidki} that the groups $G_{p,d}$ admit
faithful finite-state representations of degree $p$, independently of $d$.
In contrast, as we will prove, if $d\geq 2$, a necessary condition for the
existence of faithful state-closed representations of $G_{p,d}$ is that the
degree of the representation be a composite number.

State-closed representations of a general group $G$, are constructible from 
\textit{similarity pairs} $\left( H,f\right) $ where $H$ is a subgroup of $G$
of finite index $m$ and $f$ is a homomorphism $H\rightarrow G$ called a 
\textit{virtual endomorphism} of $G$. A similarity pair $(H,f)$ leads to a
recursively defined representation $\varphi $ of $G$ as a group of
automorphisms of the $1$-rooted regular $m$-tree. The image $G^{\varphi }$
is a state-closed group of automorphisms of the tree. The kernel of $\varphi 
$, called the $f$-core of $H$, is the largest subgroup $K$ of $H$ which is
normal in $G$ and $f$-invariant (in the sense $K^{f}\leq K$). When the
kernel of $\varphi $ is trivial, $f$ and the similarity pair $(H,f)$ are
said to be \textit{simple}. A typical example of a group with a simple
similarity pair is that of the free abelian group $X=\left\langle
x_{1},x_{2},...,x_{d}\right\rangle $ of rank $d$ and pair $\left( Y,f\right) 
$ with $Y=\left\langle x_{1}^{p},x_{2},..,x_{d}\right\rangle $ and
homomorphism $f$ which is an extension of $x_{1}^{p}\rightarrow x_{2},$ $%
x_{j}\rightarrow x_{1+j}$ ( $2\leq j\leq d-1$ ), $x_{d}\rightarrow x_{1}$
(see, \cite{NekSid}).

State-closed representations are known for many finitely generated groups
ranging from the torsion groups of Grigorchuk and Gupta-Sidki to free groups 
\cite{Nek}. Furthermore, such representations have been studied for the
family of abelian groups \cite{BruSid}, of finitely generated nilpotent
groups \cite{BerSid}, as well as for arithmetic groups \cite{Kapo}. A useful
software for computation in self similar groups is available in \cite%
{MuntSav}.

Section 2 of this paper is a preliminary analysis of similarity pairs for
groups which are semidirect products $G=AX$ where $A$ is a self-centralizing
abelian normal subgroup. We show how to replace a simple similarity pair $%
\left( H,f\right) $ where $f:A_{0}\left( =A\cap H\right) \rightarrow A$ by a
simple $\left( \dot{H},\dot{f}\right) $ satisfying 
\begin{eqnarray*}
\left[ G:\dot{H}\right] &=&\left[ G:H\right] ,\text{ }\dot{H}=A_{0}Y,\text{ }%
Y=AH\cap X, \\
\dot{f} &:&A_{0}\rightarrow A,Y\rightarrow X\text{.}
\end{eqnarray*}%
In addition, we provide a module theoretic formulation of state-closed
representations of $G$.

In Section 3 we prove that a state-closed representations of $G_{p,d}$,
where $H=A_{0}X$, is faithful only if $d=1$. The exceptional case occurs in
the classical representation of the lamplighter group $G=G_{2,1}$ where, in
addition, $H$ contains the commutator subgroup $G^{\prime }$.

\begin{theorem}
(Nonexistence Result) Let $G_{p,d}=C_{p}wrX$ where $C_{p}=\left\langle
a\right\rangle $ of prime order $p$ and $X$ is a free abelian group of
finite rank $d\geq 2$. Let $A$ be the normal closure of $\left\langle
a\right\rangle $, let $H$ be a subgroup of finite index in $G_{p,d}\left(
=AX\right) $ and $f:H\rightarrow G$ a homomorphism. Suppose $H$ projects
onto $X$ modulo $A$. Then $f$ is not simple.
\end{theorem}

An application of this result is

\begin{corollary}
If $d\geq 2$, then $G_{p,d}$ does not admit faithful state-closed
representations of degree $p$.
\end{corollary}

Previously, it was shown that finitely generated torsion-free nilpotent
groups of nilpotency class $c>1$ do not admit faithful state-closed
representations of degree $p$ \cite{BerSid}.

In Section 4, we study representations of $G_{p,1}$ on the $p$-adic tree.
First, we characterize the faithful ones obtained with respect to normal
subgroups of index $p$.

\begin{theorem}
(Degree 1 Representations) Suppose $H$ is a normal subgroup of $G_{p,1}$ of
index $p$. Then every faithful state-closed representations of $G$ on the $p$%
-adic tree obtained with respect to $H$ is reducible to 
\begin{eqnarray*}
\varphi &:&a\rightarrow \alpha =\left( 0,1,...,p-1\right) \text{,} \\
x &\rightarrow &\xi =\left( \xi ^{n},\xi ^{n}\alpha ^{u\left( \xi \right)
},...,\xi ^{n}\alpha ^{u\left( \xi \right) (p-1)}\right)
\end{eqnarray*}%
for some integer $n$ and $u\left( x\right) \in k\left\langle x\right\rangle $
such that $\gcd \left( p,n\right) =1$ and $u\left( 1\right) \not=0$.
\end{theorem}

The methods of reduction use the replacement arguments discussed in Section
2. Next, we produce faithful representations from those subgroups $H$ of
index $p$ which are not necessarily normal.

\begin{theorem}
Let $G_{p,1}=C_{p}wrC$, $C_{p}=\left\langle a\right\rangle $, $%
C=\left\langle x\right\rangle $. Also, let $j\in \left\{ 1,...,p-1\right\} $
and $\beta $ be the permutation of $\left\{ 0,1,...,p-1\right\} $ defined by 
$\beta :i\rightarrow ij$ modulo $p$. Then 
\begin{eqnarray*}
\varphi &:&a\rightarrow \alpha =\left( 0,1,...,p-1\right) , \\
x &\rightarrow &\xi =\left( \xi ,...,\xi \alpha ^{i},...,\xi \alpha ^{\left(
p-1\right) }\right) \beta \text{ }
\end{eqnarray*}%
extends to a faithful state-closed representation of $G_{p,1}$ on the $p$%
-adic tree.
\end{theorem}

We note that this representation is finite state; indeed the group $%
G^{\varphi }$ is an automaton group generated by the $p$ states of $\xi $.
For $p=2$, the representation is defined by $\varphi :a\rightarrow \alpha
=\left( 0,1\right) ,x\rightarrow \xi =\left( \xi ,\xi \alpha \right) $ and
is the classical representation of $G_{2,1}$. Also, for $p=3$, $\varphi
:a\rightarrow \alpha =\left( 0,1,2\right) ,x\rightarrow \xi =\left( \xi ,\xi
\alpha ,\xi \alpha ^{2}\right) \left( 1,2\right) $ and $\xi $ is equivalent
to the automaton in \cite{BonDaRo}.

In Section 4, we provide uniformly faithful state-closed representations of $%
G=G_{p,d}$ on the $p^{2}$-tree, for all primes $p$ and for all $d\geq 2$.

\begin{theorem}
(Degree 2 Representations) Let $d\geq 2,G=G_{p,d}$ and $G^{\prime }$ be its
commutator subgroup. Furthermore, let $H=G^{\prime }Y$ where $Y=\left\langle
x_{1}^{p},x_{2},..,x_{d}\right\rangle $. Then the map 
\begin{eqnarray*}
f &:&a^{x_{1}^{i}-1}\rightarrow a^{i}\text{ \ (}1\leq i\leq p-1\text{), }%
a^{z-1}\rightarrow e\text{ for all }z\in Y, \\
x_{1}^{p} &\rightarrow &x_{2},\text{ }x_{j}\rightarrow x_{1+j}\text{ ( }%
2\leq j\leq d-1\text{ ), }x_{d}\rightarrow x_{1}
\end{eqnarray*}%
extends to a simple homomorphism $f:H\rightarrow G_{p,d}$.
\end{theorem}

Finally, we write down concretely the above representation for $G_{2,2}$ in
its action on the $4$-tree which is indexed by sequences from $\left\{
0,1,2,3\right\} $.

\begin{theorem}
The following automorphisms of the $4$-tree 
\begin{eqnarray*}
\alpha &=&\left( 0,1\right) \left( 2,3\right) , \\
\xi _{1} &=&\left( 1,a^{\xi _{2}^{-1}},\xi _{2},a^{\xi _{2}^{-1}}\xi
_{2}\right) \left( 0,2\right) \left( 1,3\right) , \\
\xi _{2} &=&\left( \xi _{1},\xi _{1},\xi _{1},a^{\left( 1+\xi
_{1}^{-1}\right) \xi _{2}^{-1}}\xi _{1}\right) \text{.}
\end{eqnarray*}%
generate a group $G$ isomorphic to $G_{2,2}$. In addition, $G$ is the
state-closure of $\xi _{1}$ and is finite-state; indeed, $\xi _{1}$ has $12$
states.
\end{theorem}

\section{Virtual endomorphisms of semidirect products}

\subsection{Replacement arguments}

We consider in this section groups $G$ which are semidirect products $G=AX$
where $A$ is abelian, $C_{X}\left( A\right) =1$, $H$ a subgroup of $G$ of
index $m$ and homomorphism $f:H\rightarrow G$ such that $f:A_{0}\left(
=H\cap A\right) \rightarrow A$; the last condition is necessary in case $%
G=G_{p,d}$. 

We recall that a subgroup $K$ of $G$ is said to be $f$-invariant provided $%
K^{f}\leq K$.

\begin{proposition}
Let $\left( G,H,f\right) $ be as above, $Y=AH\cap X$ and let $\gamma $ be an
automorphism of $X$. Define the subgroup $\dot{H}=A_{0}Y$ of $G$. Then: (1) $%
\ \left[ G:\dot{H}\right] =\left[ G:H\right] $; (2) $f$ induces a
homomorphism $\alpha :Y\rightarrow X$ which together with $%
f|_{A_{0}}:A_{0}\rightarrow A$ defines a homomorphism $\dot{f}:\dot{H}%
\rightarrow G$; (3) $\left( H,f\right) $ may be replaced by $\left( \dot{H},%
\dot{f}\right) $ and by $\left( H^{\gamma },f^{\gamma }\right) $;(3) both
replacements preserve normality of $H$ and simplicity of $f.$
\end{proposition}

\begin{proof}
(\textbf{1}) The subgroup $A_{0}$ is normal in $H$ and $\left[ A:A_{0}\right]
=m_{1}$ a divisor of $m$. Also, $Y$ is a subgroup of $X$ and $\left[ X:Y%
\right] =m_{2}$ a divisor of $m$. Let $S$ be a right transversal of $A_{0}$
in $A$ and let $T$ be a right transversal of $Y$ in $X$. For every $y\in Y$,
there exists $v(y)\in A$ such that $H=A_{0}\left\langle v(y)y\mid y\in
Y\right\rangle $. The set $ST$ is a right transversal of $H$ in $G$ and $%
m=m_{1}m_{2}$:%
\begin{eqnarray*}
HST &=&A_{0}\left\langle v(y)y\mid y\in Y\right\rangle ST=\left\langle
v(y)y\mid y\in Y\right\rangle \left( A_{0}S\right) T \\
&=&\left\langle v(y)y\mid y\in Y\right\rangle AT=A\left\langle v(y)y\mid
y\in Y\right\rangle T \\
&=&A\left( YT\right) =AX=G\text{.}
\end{eqnarray*}

As $Y$ normalizes $A_{0}$ the set $\dot{H}=A_{0}Y$ is a subgroup of $G$ with
the same transversal $ST$.

(\textbf{2}) For every $y\in Y$ there exist a unique pair $w\left( y\right)
\in A,y^{\prime }\in X$ such that $f:v(y)y\rightarrow w\left( y\right)
y^{\prime }$. Define $\alpha :Y\rightarrow X$ by $y\rightarrow y^{\prime }$;
then $\alpha $ is a homomorphism.

Now let $\dot{f}:\dot{H}\rightarrow G$ be an extension of $\mu \left(
=f|_{A_{0}}\right) :A_{0}\rightarrow A$ and $\alpha :Y\rightarrow X$, by $%
\left( a_{0}y\right) ^{\dot{f}}=\left( a_{0}\right) ^{\mu }y^{\alpha }$ for
all $a_{0}\in A_{0},y\in Y$. To prove that $\dot{f}$ is a homomorphism it
suffices to prove $\left( \left( a_{0}\right) ^{\mu }\right) ^{y^{\alpha
}}=\left( \left( a_{0}\right) ^{y}\right) ^{\mu }$ for all $a_{0}\in
A_{0},y\in Y$:

\begin{eqnarray*}
\left( \left( a_{0}\right) ^{y}\right) ^{\mu } &=&\left( \left( a_{0}\right)
^{v(y)y}\right) ^{f}=\left( \left( a_{0}\right) ^{f}\right) ^{\left(
v(y)y\right) ^{f}} \\
&=&\left( \left( a_{0}\right) ^{\mu }\right) ^{w\left( y\right) y^{\alpha
}}=\left( \left( a_{0}\right) ^{\mu }\right) ^{y^{\alpha }}\text{.}
\end{eqnarray*}

(\textbf{3.1}) Suppose $H$ is a normal subgroup of $G$; then, $A_{0}=A\cap H$
is normal in $G$. Since $Y\leq X$ and $Y$ is normal in $X$ modulo $A$, it
follows that $Y$ is normal in $X$. Therefore, for $a_{0}\in A_{0},a\in A$, $%
y\in Y$ we have%
\begin{equation*}
\left( a_{0}y\right) ^{x}=\left( a_{0}\right) ^{x}y^{x},\text{ }\left[ y,a%
\right] =\left[ v(y)y,a\right]
\end{equation*}%
from which it follows that $\dot{H}$ is normal in $G$.

(\textbf{3.2}) Suppose $f$ is simple. To prove that $\dot{f}:\dot{H}%
\rightarrow G$ is simple, we consider a subgroup $K\leq \dot{H}$ such that $%
K $ is normal in $G$ and $K^{\dot{f}}\leq K$. Then $K_{0}=K\cap A$ is a
normal subgroup of $G$ and is $\dot{f}$-invariant :%
\begin{eqnarray*}
\left( K_{0}\right) ^{\dot{f}} &=&\left( K_{0}\right) ^{f}=\left( K\cap
A_{0}\right) ^{f}\leq K^{\dot{f}}\cap A \\
&\leq &K\cap A=\left( K\cap \dot{H}\right) \cap A \\
&=&K\cap \left( \dot{H}\cap A\right) =K\cap A_{0}=K_{0}\text{.}
\end{eqnarray*}%
Thus, $K_{0}\left( =K\cap A\right) =1$ and from $C_{X}\left( A\right) =1$,
we conclude $K=1$.

(\textbf{3.3) }The assertions about the replacement $\left( H,f\right) $ by $%
\left( H^{\gamma },f^{\gamma }\right) $ are easily verified.
\end{proof}

\begin{remark}
As $H$ is obtained from $\dot{H}=A_{0}Y$ by multiplying the elements $y\in Y$
by certain elements $v(y)\in \left( A\backslash A_{0}\right) \cup \left\{
e\right\} $, the subgroup $H$ may be viewed as a \textit{deformation} of $%
\dot{H}$ (maintaining the subgroup $A_{0}$ unchanged). Suppose $Y$ is
finitely presented, say with generators $y_{i}$ and relators $r_{j}$. Then
the process of obtaining all deformations of $\dot{H}$ is algorithmic: let $%
T $ be a transversal of $A_{0}$ in $A$ then for every choice $v_{i}\in T$ $%
\left( 1\leq i\leq s\right) $ we need that $r_{j}\left( v_{i}y_{i}\right)
\in A_{0}$ $\left( 1\leq i\leq s\right) $ be satisfied. In the case of $%
G_{p,d}$, the subgroup $Y$ is freely generated by $\left\{
y_{1},...,y_{d}\right\} $, satisfying $\left[ y_{i},y_{j}\right] =e$ for $%
1\leq i<j\leq d$. Then, since%
\begin{equation*}
\left[ v_{i}y_{i},v_{j}y_{j}\right] =\left[ w_{i},y_{j}\right] \left[
y_{i},w_{j}\right]
\end{equation*}%
where $w_{i}=v_{i}^{^{y_{i}}}$ $\left( 1\leq i\leq d\right) $, the set $%
\left\{ w_{1},...,w_{s}\right\} $ should satisfy%
\begin{equation*}
\left[ y_{i},w_{j}\right] \in \left[ y_{j},w_{i}\right] A_{0}\text{ for all }%
i,j\text{.}
\end{equation*}
\end{remark}

\subsection{Translation to module theoretic language}

We continue with the above context: $G$ is a semidirect product $G=AX$ where 
$A$ is abelian such that $C_{X}\left( A\right) =1$, $H$ \ is a subgroup of $%
G $ of finite index $m$ and $f:H\rightarrow G$ is a homomorphism such that $%
f:A_{0}\rightarrow A,Y\rightarrow X$. On viewing $A$ additively, our
conditions have the following module theoretic translation:

let $\mathcal{A}$ be the group ring $\mathbb{Z}\left( X\right) $ and $%
\mathcal{B}$ the subring$\mathcal{\ }\mathbb{Z}\left( Y\right) $. Then, $A$
is a right $\mathcal{A}$-module and $A_{0}$ a right $\mathcal{B}$-module of
index $m_{1}$ in $A$. The homomorphism $f$ induces a pair of homomorphisms $%
\left( \mu ,\alpha \right) $ where $\mu :A_{0}\rightarrow A$ is an additive
homomorphism and $\alpha $ is a multiplicative group homomorphism $\alpha
:Y\rightarrow X$.

Furthermore, $\alpha $ extends to a ring homomorphism $\alpha :\mathcal{B}%
\rightarrow \mathcal{A}$ and the action of $\mathcal{B}$ on $A_{0}$
satisfies the skew condition 
\begin{equation*}
\left( a_{0}.w\right) ^{\mu }=\left( a_{0}\right) ^{\mu }w^{\alpha }
\end{equation*}%
for all $a_{0}\in A_{0}$ and $w\in \mathcal{B}$.

\begin{lemma}
With the above notation, suppose $A_{0}$ is normal in $G$. Suppose $K$
normal $f$-invariant subgroup of $G$ implies $K\cap A=1$. Then, $f$ is a
simple$.$ Equivalently, in additive terms, if $f$ is simple then the only
right $\mathcal{A}$-submodule of $A$ contained in $A_{0}$ which is also $\mu 
$-invariant is $\left\{ 0\right\} $.
\end{lemma}

\begin{proof}
The subgroup $K_{0}=K\cap A$ is normal in $G$ and $K_{0}\leq K\cap A_{0}$.
Furthermore, $\left( K_{0}\right) ^{f}\leq \left( K\cap A_{0}\right)
^{f}\leq K\cap A=K_{0}$. We conclude, $K_{0}=1$ and $\left[ K,A\right] =1$
and $K\leq C_{X}\left( A\right) =1$.
\end{proof}

\begin{remark}
For the groups $G=G_{p,d}$, the ring $\mathcal{A}$ is the commutative group
algebra $k\left\langle X\right\rangle $ where $k=GF\left( p\right) $ and $%
\mathcal{B=}k\left[ Y\right] $. Suppose $A_{0}$ is a normal subgroup of $%
G_{p,d}$. Then$\ A_{0}$\ corresponds to an ideal $\mathcal{I}$ of $\mathcal{A%
}$. If $v$ is the exponent of the group of units of the quotient ring $\frac{%
\mathcal{A}}{\mathcal{I}}$ then $\mathcal{I}$ contains the ideal $\mathcal{A}%
\left( v\right) =\sum_{1\leq i\leq d}\mathcal{A}\left( x_{i}^{v}-1\right) $.
Moreover, the skew action reads as%
\begin{equation}
\text{ }\left( w.\nu \right) ^{\mu }=w^{\alpha }\left( \nu \right) ^{\mu } 
\tag{*}
\end{equation}%
for all $\nu \in \mathcal{I}$ and $w\in \mathcal{B=}\mathbb{Z}\left[ Y\right]
$. On denoting $K=\ker _{\mathcal{B}}\left( \alpha \right) $, it follows
that $\mathcal{I}K$ is an ideal of $\mathcal{A}$ contained in $\ker _{%
\mathcal{I}}\left( \mu \right) $. If $f$ is simple then the only $\mu $%
-invariant ideal of $\mathcal{A}$ is the zero ideal and therefore, $\alpha :$
$Y\rightarrow X$ is a monomorphism.
\end{remark}

\section{Nonexistence of Simple Virtual Endomorphisms}

The following result on solving certain polynomial equations will be used in
the proof of Theorem 1.

\begin{proof}
\end{proof}

\begin{proposition}
Let $M=\left( m_{ij}\right) $ be an $n\times n$ integral matrix where $n\geq
2$ and $v$ a nonzero integer; denote $\det \left( M\right) =t$. Let $K$ be a
field and let $u_{i}\in $ $K\left\langle x_{1},x_{2},...,x_{n}\right\rangle $
satisfy the equations 
\begin{equation*}
\left( x_{1}^{vm_{i1}}x_{2}^{vm_{i2}}...x_{n}^{vm_{in}}-1\right)
u_{j}=\left( x_{1}^{vm_{j1}}x_{2}^{vm_{j2}}...x_{n}^{vm_{jn}}-1\right) u_{i}
\end{equation*}%
for all $1\leq i<j\leq n$. Then, either $t=0$ or $u_{i}\in \mathcal{I}%
=\left\langle x_{1}^{v}-1,x_{2}^{v}-1,...,x_{d}^{v}-1\right\rangle _{\text{%
ideal}}$ for all $i$.
\end{proposition}

\begin{proof}
(1) Let $n=2$. Suppose $g\in k\left\langle x_{1},x_{2}\right\rangle $ is a
nonconstant common factor of $x_{1}^{vm_{11}}x_{2}^{vm_{12}}-1$, $%
x_{1}^{vm_{21}}x_{2}^{vm_{22}}-1$. Then, $x_{1}^{vm_{11}}x_{2}^{vm_{12}}%
\equiv 1,x_{1}^{vm_{21}}x_{2}^{vm_{22}}\equiv 1$ $\func{mod}g$. Therefore,  
\begin{eqnarray*}
\left( x_{1}^{vm_{11}}x_{2}^{vm_{12}}\right) ^{m_{22}}\left(
x_{1}^{vm_{21}}x_{2}^{vm_{22}}\right) ^{-m_{12}} &\equiv &1 \\
&\equiv
&x_{1}^{vm_{11}m_{22}-vm_{21}m_{22}}x_{2}^{vm_{12}m_{22}-vm_{12}vm_{22}} \\
&\equiv &x_{1}^{vm_{11}m_{22}-vm_{21}m_{22}}\equiv x_{1}^{vt}\text{ }\func{%
mod}g\text{;}
\end{eqnarray*}

similarly, 
\begin{equation*}
x_{2}^{vt}\equiv 1\func{mod}g\text{.}
\end{equation*}
That is, there exist $h,h^{\prime }\in k\left\langle
x_{1},x_{2}\right\rangle $ such that%
\begin{equation*}
x_{1}^{vt}-1=gh,\text{ }x_{2}^{vt}-1=gh^{\prime }\text{;}
\end{equation*}%
the left hand of each equation is non-zero. Since $K\left[ x_{1},x_{2}\right]
$ is a unique factorization domain, $g,h\in k\left\langle x_{1}\right\rangle 
$ and $g,h^{\prime }\in k\left\langle x_{2}\right\rangle $; therefore, $g\in
k$; contradiction.

We conclude $\left( x_{1}^{vm_{21}}x_{2}^{vm_{22}}-1\right) |u_{1}$, $\left(
x_{1}^{vm_{11}}x_{2}^{vm_{12}}-1\right) |u_{2}$ and thus, there exists $%
l\left( x_{1},x_{2}\right) \in k\left[ x_{1},x_{2}\right] ~$such that%
\begin{equation*}
u_{1}=\left( x_{1}^{vm_{21}}x_{2}^{vm_{22}}-1\right) l\left(
x_{1},x_{2}\right) ,\text{ }u_{2}=\left(
x_{1}^{vm_{11}}x_{2}^{vm_{12}}-1\right) l\left( x_{1},x_{2}\right) \text{.}
\end{equation*}%
Hence $u_{1},u_{2}\in \left\langle x_{1}^{v}-1,x_{2}^{v}-1\right\rangle _{%
\text{ideal}}$.

(2) Let $n\geq 3$ and suppose by induction the assertion true for $n-1$.
Suppose $\ $furthermore $u_{s}\not\in \mathcal{I}$ for some $s$; without
loss, $s=1$. Let $M_{i,j}$ denote the $\left( i,j\right) $th minor of $M$.
On letting $x_{l}=1$ ($1\leq l\leq n$) in%
\begin{equation*}
\left( x_{1}^{vm_{11}}x_{2}^{vm_{12}}...x_{n}^{vm_{1n}}-1\right)
u_{i}=\left( x_{1}^{vm_{i1}}x_{2}^{vm_{i2}}...x_{n}^{vm_{in}}-1\right) u_{1}%
\text{ }\left( i\not=1\right)
\end{equation*}%
we produce by induction $M_{i,j}=0$ for all $\left( i,j\right) $ where $%
i\not=1$. Therefore 
\begin{eqnarray*}
\det \left( M\right) &=&\sum_{1\leq i\leq n}\left( -1\right)
^{i+j}m_{i,j}M_{i,j}=\left( -1\right) ^{1+j}m_{1,j}M_{1,j}\text{ for }\left(
1\leq j\leq n\right) \text{,} \\
\sum_{1\leq j\leq n}\det \left( M\right) &=&\sum_{1\leq j\leq n}\left(
-1\right) ^{1+j}m_{1,j}M_{1,j}=\det \left( M\right) , \\
\left( n-1\right) \det \left( M\right) &=&0\text{, }\det \left( M\right) =0%
\text{;}
\end{eqnarray*}%
contradiction.
\end{proof}

\subsection{Proof of Theorem 1}

The subgroup $A_{0}$ is normal in $G$, since it is centralized by $A$ and is
normalized by $X$. In additive notation, the subgroup $A_{0}$ is an ideal $%
\mathcal{I}$ of $\mathcal{A}=k\left[ X\right] $. To prove Theorem 1, the
following proposition suffices.

\begin{proposition}
Let $X=\left\langle x_{1},x_{2},...,x_{d}\right\rangle $ be a free abelian
group with $d\geq 2$. Furthermore let $\mathcal{A}$ be the group algebra $k%
\left[ X\right] $ and let $\mathcal{I}$ be an ideal in $\mathcal{A}$ of
finite index $m>1$. Let $\mu :\mathcal{I\rightarrow A},$ be a $k$%
-homomorphism and $\alpha :X\rightarrow X$ a multiplicative homomorphism
such that the skew action 
\begin{equation}
\left( r\nu \right) ^{\mu }=r^{\alpha }\nu ^{\mu }  \tag{**}
\end{equation}%
is satisfied for all $r\in \mathcal{A}$, $\nu \in \mathcal{I}$. $\ $Then 
\textit{the ideal }$\mathcal{I}$ contains a nontrivial\textit{\ }$\mu $%
\textit{-invariant ideal} of $\mathcal{A}$.
\end{proposition}

\begin{proof}
We suppose by contradiction that there exists a simple homomorphism $\mu :%
\mathcal{I\rightarrow A}$. Then, by Remark 2, $\alpha $ is a monomorphism.

Let the order of the group of units of $\frac{\mathcal{A}}{\mathcal{I}}$ be $%
v$. Then $\mathcal{J=}\left\langle X^{v}-1\right\rangle _{\text{ideal}}$ is
contained in $\mathcal{I}$ \ and so, we may assume $\mathcal{I=J}$.

Denote%
\begin{equation*}
\left( x_{i}^{v}-1\right) ^{\mu }=u_{i}\text{ }\left( 1\leq i\leq d\right) 
\text{.}
\end{equation*}%
Then, conditions (**) imply that for $i\not=j$, 
\begin{eqnarray*}
\mu &:&\left( x_{i}^{v}-1\right) \left( x_{j}^{v}-1\right) \rightarrow
\left( x_{i}^{v}-1\right) ^{\alpha }u_{j}\text{ } \\
&:&\left( x_{j}^{v}-1\right) \left( x_{i}^{v}-1\right) \rightarrow \left(
x_{j}^{v}-1\right) ^{\alpha }u_{i}
\end{eqnarray*}%
and therefore, 
\begin{equation*}
\left( x_{i}^{v\alpha }-1\right) u_{j}=\left( x_{j}^{v\alpha }-1\right) u_{i}%
\text{.}
\end{equation*}%
Since $\alpha :X\rightarrow X$ is a monomorphism, by the above proposition, $%
u_{i}\left( =\left( x_{i}^{v}-1\right) ^{\mu }\right) \in \mathcal{I}$ holds
for all $i$. As, $\mu :r\left( x_{i}^{v}-1\right) \rightarrow r^{\alpha
}u_{i}$ for all $i$, we conclude that $\mathcal{I}$ is $\mu $-invariant.
\end{proof}

\textbf{Proof of Corollary 1.}

\begin{proof}
Suppose $H$ is a subgroup of $G$ of prime index $q$ and $f:H\rightarrow G$ a
simple homomorphism. Then, as $f:A_{0}\rightarrow A$, it follows that $A_{0}$
is a proper subgroup of $A$ and so, $\left[ A:A_{0}\right] =q$. Therefore, $%
H $ projects onto $X$ modulo $A$ and $A_{0}$ is normal in $G$. We apply the
previous theorem to conclude the existence of a nontrivial subgroup of $%
A_{0} $ which is normal in $G$ and $f$-invariant and so reach a
contradiction.
\end{proof}

\section{Representations of $G_{p,1}$ of Degree $p$}

Given a general group $G$ and triple $\left( G,H,f\right) $ with $\left[ G;H%
\right] =m$, we recall the representation $\varphi $ of $G$ on the $m$-ary
tree indexed by sequences from $N=\left\{ 0,1,...,m-1\right\} $. Let $%
T=\left\{ e,t_{2},...,t_{m-1}\right\} $ be a right transversal $T$ of $H$ in 
$G$ and $\sigma $ be the permutational representation of $G$ on $T$. Then 
\begin{equation*}
g^{\varphi }=\left( \left( \left( h_{i}\right) ^{f}\right) ^{\varphi }\mid
0\leq i\leq m-1\right) g^{\sigma }\text{,}
\end{equation*}%
where $h_{i}$ are the Schreier elements of $H$ defined by 
\begin{equation*}
h_{i}=\left( t_{i}g\right) \left( t_{j}\right) ^{-1},\text{ }t_{j}=\left(
t_{i}\right) ^{g^{\sigma }}\text{;}
\end{equation*}%
see \cite{NekSid}.

Let $G=G_{p,1}$. We observe 
\begin{equation*}
\frac{G}{G^{\prime }}=C_{p}\times C\text{, }G=A\left\langle x\right\rangle 
\text{, }A=G^{\prime }\left\langle a\right\rangle \text{.}
\end{equation*}%
The following items deal with general state-closed representations of $%
G_{p,1}$ of degree $p$.

(\textbf{i}) Let $H$ be a subgroup of index $p$ in $G$. Then from the
representation of $G$ into $Sym\left( p\right) $ we conclude that $\frac{G}{%
Core(H)}$ is isomorphic to a metabelian transitive subgroup of $Sym\left(
p\right) $ and therefore is of order multiple of $p$ and is isomorphic to a
subgroup of the semidirect product $C_{p}C_{p-1}$ where $C_{p}$ is normal
and $C_{p-1}$ acts on it as the full automorphism group $Aut\left(
C_{p}\right) $.

As $A_{0}=A\cap H$ and $\left[ G:H\right] =p\,$, we conclude that $%
A_{0}=A\cap Core\left( H\right) $ and so is normal in $G$. As we have argued
in Section 2, we may assume $H=A_{0}X$.

(\textbf{ii}) Additively, $A$ corresponds to the $k$-algebra $\mathcal{A=}%
k\left\langle x\right\rangle $ and $A_{0}$ corresponds to a maximal ideal $%
\mathcal{M}$ of $\mathcal{A}$. As $\frac{k\left\langle x\right\rangle }{%
\mathcal{M}}\cong k$, it follows that $\mathcal{M+}x$ $=$ $\mathcal{M+}c$
where $\left( 1\leq c\leq p-1\right) $. Therefore, $\mathcal{M=}\left\langle
x-c\right\rangle _{\text{ideal }}$ and so we have $p-1$ distinct maximal
ideals%
\begin{equation*}
\mathcal{M}\text{ }\mathcal{=}\text{ }\left\langle x-c\right\rangle _{\text{%
ideal }}\text{ }\left( 1\leq c\leq p-1\right) \text{.}
\end{equation*}

(\textbf{iii}) Let $A_{0}$ correspond to $\mathcal{M}$ and let $%
f:H\rightarrow G$ be simple. Then, as before, $f$ corresponds to a pair of
homomorphisms $\left( \mu ,\alpha \right) $ where $\mu :$ $\mathcal{M}%
\rightarrow \mathcal{A}$ is an additive homomorphism and $\alpha
:X\rightarrow X$ is a multiplicative monomorphism. Let $\mu :x-c\rightarrow
u\left( x\right) $ for some $u\left( x\right) \in \mathcal{A}$, $\alpha
:x\rightarrow x^{n}$ for some integer $n$. Then $\mu :\mathcal{M\rightarrow A%
}$ is defined by%
\begin{equation*}
\mu :r\left( x\right) \left( x-c\right) \rightarrow r\left( x^{n}\right)
u\left( x\right) .
\end{equation*}%
We derive below some restrictions on $n,u\left( x\right) $.

(\textbf{iii.1}) \textbf{Assertion}. $u\left( c^{n^{i}}\right) \not=0$ for
all $i$.

\textbf{Proof}. If $\left( x-c\right) |u$ then $\mu :\mathcal{M}\rightarrow 
\mathcal{M}$; contradiction.

Suppose $u\left( x^{n^{t}}\right) $ is not a multiple of $x-c$ for $0\leq
t\leq i-1$, yet $u\left( x^{n^{i}}\right) =u^{\prime }\left( x\right) \left(
x-c\right) $. Then 
\begin{eqnarray*}
\mu  &:&u\left( x\right) u\left( x^{n}\right) ...u\left( x^{n^{i-1}}\right)
\left( x-c\right) \rightarrow u\left( x^{n}\right) ...u\left(
x^{n^{i}}\right) u\left( x\right)  \\
&=&\left( u\left( x\right) u\left( x^{n}\right) ...u\left(
x^{n^{i-1}}\right) \right) u^{\prime }\left( x\right) \left( x-c\right) 
\text{;}
\end{eqnarray*}%
that is, the ideal $\mathcal{M}u\left( x\right) u\left( x^{n}\right)
...u\left( x^{n^{i-1}}\right) $ is $\mu $-invariant; contradiction.

(\textbf{iii.2}) Let $n=p^{s}n^{\prime },$ $\gcd \left( p,n^{\prime }\right)
=1,$ $n^{\prime }n^{\prime \prime }\equiv 1~\func{mod}p$. Then, $s=0$ or $%
o\left( c\right) $ $\nmid $ $n^{\prime }-1$.

\textbf{Proof}. Suppose $s\not=0$ and $o\left( c\right) |n^{\prime }-1\,$.
Then,

\begin{eqnarray*}
\mu &:&r\left( x\right) \left( x-c\right) ^{2}=\left( r\left( x\right)
\left( x-c\right) \right) .\left( x-c\right) \rightarrow r\left( x\right)
^{\alpha }(x-c)^{\alpha }.\left( x-c\right) ^{\mu } \\
&=&r\left( x^{n}\right) \left( x^{n}-c\right) u\left( x\right) =r\left(
x^{n}\right) \left( x^{n^{\prime }}-c\right) ^{p^{s}}u\left( x\right) \\
&=&r\left( x^{n}\right) \left( x^{n^{\prime }}-c^{n^{\prime }}\right)
^{p^{s}}u\left( x\right) =\ast \left( x-c\right) ^{2}\text{;}
\end{eqnarray*}%
thus $\mathcal{M}\left( x-c\right) ^{2}$ is $\mu $-invariant; a
contradiction.

\subsection{\textbf{Proof of Theorem 2}}

(\textbf{1}) Since $\ H$ contains $G^{\prime }$, it follows that $G^{\prime
}=A_{0}$. By the replacement argument in Section 2, we may assume $%
H=A_{0}\left\langle x\right\rangle $\ and $f:x\rightarrow x^{n}$. for some
nonzero integer. Furthermore, additively, $A_{0}$ corresponds to the ideal $%
\mathcal{I=}\left\langle x-1\right\rangle _{\text{ideal}}$ in the group
algebra $\mathcal{A=}k\left\langle x\right\rangle $. Therefore, $f$ $\ $%
induces on $\mathcal{I}$ the additive homomorphism $\mu :\mathcal{%
I\rightarrow A}$ defined by 
\begin{equation*}
\mu :r\left( x\right) \left( x-1\right) \rightarrow r\left( x^{n}\right)
u\left( x\right)
\end{equation*}%
for all $r\left( x\right) \in \mathcal{A}$, where $u\left( x\right) $ is a
fixed non-zero element of $\mathcal{A}$. Since $\mathcal{I}$ is not $\mu $%
-invariant, $\gcd \left( -1+x,u\right) =1$; that is, $u\left( 1\right)
\not=0 $.

(\textbf{2}) \textbf{Assertion}: $\gcd \left( p,n\right) =1$.

\textbf{Proof. }Suppose $n=pn^{\prime }$. Then, as in the previous (iii.2),%
\begin{eqnarray*}
\mu &:&r\left( x\right) \left( x-1\right) .\left( x-1\right) \rightarrow
\left( r\left( x\right) \left( x-1\right) \right) ^{\alpha }.\left(
x-1\right) ^{\mu } \\
&=&\left( r\left( x^{n}\right) \left( x^{n}-1\right) \right) .u\left(
x\right) \\
&=&r\left( x^{n}\right) \left( x^{n^{\prime }}-1\right) ^{p}u\left( x\right)
=t\left( x\right) \left( x-1\right) ^{2}
\end{eqnarray*}%
which proves that $\mathcal{K=I}\left( x-1\right) ^{2}$ is $\mu $-invariant;
a contradiction.

(\textbf{3}) \textbf{Assertion}: $\mu $ is simple.

\textbf{Proof}. Denote $\frac{x^{n}-1}{x-1}$ by $\Phi _{n}\left( x\right) $.
Let $v=\left( x-1\right) ^{j}r\left( x\right) $ be a non-zero polynomial in $%
\mathcal{I}$ such that $j\geq 1$, $r\left( 1\right) \not=0$. We have%
\begin{eqnarray*}
\mu &:&\nu =\left( x-1\right) ^{j}r\left( x\right) =\left( x-1\right)
^{j-1}r\left( x\right) \left( x-1\right) \\
&\rightarrow &\nu _{1}=\left( x^{n}-1\right) ^{j-1}r\left( x^{n}\right)
u\left( x\right) \text{,} \\
&=&\left( x-1\right) ^{j-1}\left( \Phi _{n}\left( x\right) ^{j-1}u\left(
x\right) r\left( x^{n}\right) \right) ,
\end{eqnarray*}%
where%
\begin{equation*}
\Phi _{n}\left( 1\right) ^{j-1}u\left( 1\right) r\left( 1\right) \not=0\text{%
.}
\end{equation*}%
Thus, $\mu ^{j}:v\rightarrow v^{\prime }$ and $v^{\prime }\left( 1\right)
\not=0$; that is, $\left( v\right) ^{\mu ^{j}}\not\in \mathcal{I}$.

(\textbf{4}) Choose the transversal $\left\{ a^{i}\mid i=0,...,p-1\right\} $
for $H$ in $G$. Then on identifying $g$ with $g^{\varphi }$, the image of $G$
on the $p$-adic tree takes the form

\begin{equation*}
a=\left( 0,1,...,p-1\right) ,x=\left( x^{n},x^{n}a^{u\left( x\right)
},...,x^{n}a^{u\left( x\right) \left( p-1\right) }\right) \text{.}
\end{equation*}

\textbf{Proof.} Since $Ha^{i}x=Ha^{ic}$ $\left( 0\leq i\leq p-1\right) $,
the cofactor is $a^{i}xa^{-ic}=xa^{ix}a^{-ic}=xa^{i\left( x-c\right) }\in H$%
. Thus, 
\begin{equation*}
x^{\sigma }:i\rightarrow ic,\text{ }f\text{ }:xa^{i\left( x-c\right)
}\rightarrow x^{n}a^{iu(x)}\text{.}
\end{equation*}

\subsection{\textbf{Proof of Theorem 3}}

Recall Let $u\left( x\right) =1,n=1$.

(\textbf{1}) \textbf{Assertion} The homomorphism $f$ $:H\rightarrow G$ is
simple.

\textbf{Proof.} It is sufficient to prove the induced $\mu :\mathcal{%
I\rightarrow }\mathcal{A}$ is simple. A non-zero polynomial in $\mathcal{I}$
can be written as $\nu =r\left( x\right) \left( x-c\right) ^{j}$ with $j\geq
1$, $r\left( c\right) \not=0$.

We have%
\begin{eqnarray*}
\mu &:&\nu \left( x\right) =r\left( x\right) \left( x-c\right) ^{j}=r\left(
x\right) \left( x-c\right) ^{j-1}.\left( x-c\right) \\
&\rightarrow &\nu _{1}\left( x\right) =r\left( x\right) \left( x-c\right)
^{j-1}\text{.}
\end{eqnarray*}%
Thus, $\mu ^{j}:v\rightarrow v^{\prime }$ where $v^{\prime }\left( 1\right)
\not=0$.

(\textbf{2}) Choose the transversal $\left\{ a^{i}\mid i=0,...,p-1\right\} $
for $H$ in $G$. Then on identifying $g$ with $g^{\varphi }$, the
representation $\varphi $ of $G$ on the $p$-adic tree takes the form 
\begin{eqnarray*}
a &=&\left( 0,1,...,p-1\right) ,\text{ }x=\left( x,...,xa^{i},...,xa^{\left(
p-1\right) }\right) x^{\sigma }, \\
\text{ where }x^{\sigma } &:&i\rightarrow ic.
\end{eqnarray*}

\section{Representations of $G_{p,d}$ ($d\geq 2$) of Degree $p^{2}$}

If $\mathcal{C}$ is a group algebra, we denote its augmentation ideal by $%
\mathcal{C}^{\prime }$.

Let 
\begin{eqnarray*}
G &=&G_{p,d},\text{ }H=G^{\prime }Y,\text{ } \\
Y &=&\left\langle x_{1}^{p},x_{2},...,x_{d}\right\rangle , \\
X_{2} &=&\left\{ x_{2},...,x_{d}\right\} ,Z=\left\langle X_{2}\right\rangle 
\text{.}
\end{eqnarray*}%
Then, $\left[ G;H\right] =p^{2}$ and $A_{0}=A\cap H=G^{\prime }$. Also let $%
\alpha :Y\rightarrow X$ be the monomorphism defined by 
\begin{equation*}
x_{1}^{p}\rightarrow x_{2},\text{ }x_{j}\rightarrow x_{1+j}\text{ ( }2\leq
j\leq d-1\text{ ), }x_{d}\rightarrow x_{1}\text{.}
\end{equation*}

Let $\mathcal{A}=k\left( X\right) ,\mathcal{B}=k\left( Y\right) $; then, $%
\mathcal{A}=\sum_{0\leq i\leq p-1}\mathcal{B}x^{i}$. Also, $\mathcal{A}%
^{\prime }$ is the same as the ideal $\mathcal{I}$ generated by $\left\{
x_{1}-1,x_{2}-1,...,x_{d}-1\right\} $. We continue with: $A$ corresponds to $%
\mathcal{A}$ and $A_{0}$ to $\mathcal{I}$.

Given an integer $j$, write $j=j_{0}+j_{1}p$, $0\leq j_{0}\leq p-1$. Then,

\begin{equation*}
k\left\langle x\right\rangle ^{\prime }=k\left\langle x^{p}\right\rangle
^{\prime }\oplus \sum_{1\leq i\leq p-1}k\left\langle x^{p}\right\rangle
\left( x^{i}-1\right) \text{.}
\end{equation*}

\textbf{(1)} \textbf{Decomposition of\ }$\mathcal{I}$.

The ideal $\mathcal{I}$ decomposes as \textbf{\ } 
\begin{eqnarray*}
\mathcal{I} &\mathcal{=}&\sum \left\{ k\left( xz-1\right) \mid x\in
\left\langle x_{1}\right\rangle ,z\in Z\right\} \\
&=&\sum_{x\in \left\langle x_{1}\right\rangle }k\left( x-1\right) \oplus
\sum_{z\in Z}k\left( z-1\right) \oplus \sum_{x\in \left\langle
x_{1}\right\rangle ,z\in Z}k\left( x-1\right) \left( z-1\right)
\end{eqnarray*}%
and on substituting 
\begin{equation*}
\sum_{x\in \left\langle x_{1}\right\rangle }k\left( x-1\right)
=k\left\langle x_{1}^{p}\right\rangle ^{\prime }\oplus \sum_{1\leq i\leq
p-1}k\left\langle x_{1}^{p}\right\rangle \left( x_{1}^{i}-1\right) \text{,}
\end{equation*}%
the decomposition is refined to 
\begin{eqnarray*}
\mathcal{I} &=&k\left\langle x_{1}^{p}\right\rangle ^{\prime }\oplus
\sum_{1\leq i\leq p-1}k\left\langle x_{1}^{p}\right\rangle \left(
x_{1}^{i}-1\right) \\
&&\oplus \sum_{z\in Z}k\left( z-1\right) \oplus \sum_{z\in Z}k\left\langle
x_{1}^{p}\right\rangle ^{\prime }\left( z-1\right) \\
&&\oplus \sum_{z\in Z}\sum_{1\leq i\leq p-1}k\left\langle
x_{1}^{p}\right\rangle \left( x_{1}^{i}-1\right) \left( z-1\right) \text{.}
\end{eqnarray*}%
Therefore, an element $v$ in $\mathcal{I}$ has the unique form 
\begin{equation*}
\nu \mathcal{=}b_{0}+\sum_{1\leq i\leq p-1}b_{i}\left( x_{1}^{i}-1\right)
\oplus \sum_{z\in Z}a_{z}\left( z-1\right) \oplus \sum_{z\in Z}\sum_{0\leq
i\leq p-1}b_{i,z}\left( z-1\right) \left( x_{1}^{i}-1\right) \text{,}
\end{equation*}%
where 
\begin{equation*}
b_{0}\in k\left\langle x_{1}^{p}\right\rangle ^{\prime }\text{, }%
a_{z},b_{i},b_{i,z}\in k\left\langle x_{1}^{p}\right\rangle \text{.}
\end{equation*}

\textbf{(2)} \textbf{Definition of }$\mu $ on $\mathcal{I}$

Choose $\mu :\mathcal{B}^{\prime }\rightarrow 0$ and 
\begin{equation*}
\mu :x_{1}^{i}-1\rightarrow i\text{ for }1\leq i\leq p-1.
\end{equation*}%
Given the algebra homomorphism $\alpha :\mathcal{B}\rightarrow \mathcal{A}$,
we extend $\mu $ to an additive homomorphism $\mathcal{I\rightarrow A}$
satisfying the skew condition as follows:%
\begin{eqnarray*}
\nu &\mathcal{=}&b_{0}+\sum_{0\leq i\leq p-1}b_{i}\left( x_{1}^{i}-1\right)
+\sum_{z\in Z}a_{z}\left( z-1\right) \oplus \sum_{z\in Z}\sum_{0\leq i\leq
p-1}b_{i,z}\left( z-1\right) \left( x_{1}^{i}-1\right) \\
&\rightarrow & \\
&&\left( b_{0}\right) ^{\mu }+\sum_{0\leq i\leq p-1}\left( b_{i}\right)
^{\alpha }\left( x_{1}^{i}-1\right) ^{\mu }+\sum_{z\in Z}\left( a_{z}\right)
^{\alpha }\left( z-1\right) ^{\mu }+\sum_{z\in Z}\sum_{0\leq i\leq
p-1}\left( b_{i,z}\right) ^{\alpha }\left( z-1\right) ^{\alpha }\left(
x_{1}^{i}-1\right) ^{\mu } \\
&=&\sum_{0\leq i\leq p-1}i\left( b_{i}\right) ^{\alpha }+\sum_{z\in
Z}\sum_{0\leq i\leq p-1}i\left( b_{i,z}\right) ^{\alpha }\left( z-1\right)
^{\alpha }\text{.}
\end{eqnarray*}

\subsection{Proof of Theorem 4}

\begin{lemma}
Let $u,i\geq 1$ and write $u=u_{0}+u_{1}p,i=i_{0}+i_{1}p$ where $0\leq
u_{0},i_{0}\leq p-1$ and let $2\leq j$ $\leq d$. Then,%
\begin{eqnarray*}
\mu &:&x_{1}^{u_{0}+u_{1}p}-1\rightarrow u_{0}x_{2}^{u_{1}}, \\
x_{1}^{u}\left( x_{1}^{i}-1\right) &\rightarrow &\left( \left(
u_{0}+i_{0}\right) x_{2}^{i_{1}}-u_{0}\right) x_{2}^{u_{1}},
\end{eqnarray*}%
and%
\begin{eqnarray*}
x_{1}^{u}\left( x_{j}^{i}-1\right) &\rightarrow &u_{0}x_{2}^{u_{1}}\left(
x_{1+j}^{i}-1\right) , \\
x_{j}^{u}\left( x_{1}^{i}-1\right) &\rightarrow
&i_{0}x_{1+j}^{u}x_{2}^{i_{1}}
\end{eqnarray*}%
where $1+j$ is computed modulo $d$.
\end{lemma}

\begin{proof}
Then, 
\begin{equation*}
\mu :x_{1}^{p}\left( x_{1}-1\right) \rightarrow x_{2}.1=x_{2};
\end{equation*}%
\begin{eqnarray*}
\mu &:&x_{1}^{u_{0}+u_{1}p}-1=\left( x_{1}^{u_{1}p}-1\right) \left(
x_{1}^{u_{0}}-1\right) +\left( x_{1}^{u_{1}p}-1\right) +\left(
x_{1}^{u_{0}}-1\right) \\
&\rightarrow &\left( x_{2}^{u_{1}}-1\right) u_{0}+0+u_{0}=u_{0}x_{2}^{u_{1}}%
\text{;}
\end{eqnarray*}%
\begin{eqnarray*}
\mu &:&x_{1}^{u_{0}+u_{1}p}\left( x_{1}-1\right) =x_{1}^{\left(
u_{0}+1\right) +u_{1}p}-x_{1}^{u_{0}+u_{1}p} \\
&=&\left( x_{1}^{\left( u_{0}+1\right) +u_{1}p}-1\right) -\left(
x_{1}^{u_{0}+u_{1}p}-1\right) \\
&\rightarrow &\left( u_{0}+1\right)
x_{2}^{u_{1}}-u_{0}x_{2}^{u_{1}}=x_{2}^{u_{1}}\text{ if }u_{0}+1\leq p-1;
\end{eqnarray*}%
\begin{eqnarray*}
\mu &:&x_{1}^{\left( p-1\right) +u_{1}p}\left( x_{1}-1\right) =\left(
x_{1}^{\left( u_{1}+1\right) p}-1\right) -\left( x_{1}^{\left( p-1\right)
+u_{1}p}-1\right) \\
&\rightarrow &0-\left( p-1\right) x_{2}^{u_{1}}=x_{2}^{u_{1}}\text{;}
\end{eqnarray*}%
that is, for all $u$, 
\begin{equation*}
\mu :x_{1}^{u}\left( x_{1}-1\right) \rightarrow x_{2}^{u_{1}}\text{.}
\end{equation*}%
More generally,%
\begin{eqnarray*}
\mu &:&x_{1}^{u}\left( x_{1}^{i}-1\right) =x_{1}^{u+i}-x_{1}^{u}=\left(
x_{1}^{u+i}-1\right) -\left( x_{1}^{u}-1\right) \\
&\rightarrow &\left( u+i\right) _{0}.x_{2}^{\left( u+i\right)
_{1}}-u_{0}x_{2}^{u_{1}} \\
&=&\left( u_{0}+i_{0}\right) x_{2}^{u_{1}+i_{1}}-u_{0}x_{2}^{u_{1}}\text{ if 
}\left( u+i\right) _{0}\leq p-1; \\
\mu &:&x_{1}^{u}\left( x_{1}^{i}-1\right) \rightarrow -u_{0}x_{2}^{u_{1}}%
\text{ if }\left( u+i\right) _{0}=p\text{.}
\end{eqnarray*}%
In all cases,%
\begin{equation*}
x_{1}^{u}\left( x_{1}^{i}-1\right) \rightarrow \left( u_{0}+i_{0}\right)
x_{2}^{u_{1}+i_{1}}-u_{0}x_{2}^{u_{1}}\text{.}
\end{equation*}%
Also, for $2\leq j\leq d$,%
\begin{equation*}
x_{1}^{u}\left( x_{j}^{i}-1\right) =\left( x_{1}^{u}-1\right) \left(
x_{j}^{i}-1\right) +\left( x_{j}^{i}-1\right) \rightarrow \left(
x_{j}^{i}-1\right) ^{\alpha }\left( x_{1}^{u}-1\right) ^{\mu }=\left(
x_{j+1}^{i}-1\right) u_{0}x_{2}^{u_{1}};
\end{equation*}%
\begin{equation*}
x_{j}^{u}\left( x_{1}^{i}-1\right) \rightarrow \left( x_{j+1}^{u}\right)
i_{0}x_{2}^{i_{1}}=i_{0}x_{j+1}^{u}x_{2}^{i_{1}}\text{.}
\end{equation*}
\end{proof}

\begin{lemma}
Let $q\left( x\right) =c_{0}+c_{1}x+..+c_{s}x^{s}\in k\left[ x\right] $, and
let $0\leq u=$ $u_{0}+u_{1}p$ where $0\leq u_{0}\leq p-1$. Suppose $q\left(
x_{1}\right) \in \mathcal{I}$. Then, $\sum_{0\leq i\leq s}c_{i}=0$ and 
\begin{equation*}
\mu :q\left( x_{1}\right) \rightarrow \sum_{1\leq i\leq
s}c_{i}i_{0}x_{2}^{i_{1}}.
\end{equation*}%
Furthermore,%
\begin{equation*}
\left( x_{1}^{u}q\left( x_{1}\right) \right) ^{\mu }-x_{2}^{u_{1}}.q\left(
x_{1}\right) ^{\mu }=u_{0}\left( \sum_{1\leq i\leq s}c_{i}\left(
x_{2}^{i_{1}}-1\right) \right) x_{2}^{u_{1}}\text{.}
\end{equation*}
\end{lemma}

\begin{proof}
Let $1\leq $ $j\leq d$. Then,%
\begin{equation*}
q\left( x\right) =c_{0}+c_{1}x+..+c_{s}x^{s}=\sum_{0\leq i\leq
s}c_{i}+\sum_{1\leq i\leq s}c_{i}\left( x^{i}-1\right) \text{,}
\end{equation*}%
and clearly, $q\left( x_{1}\right) \in \mathcal{I}$ if and only if $%
\sum_{0\leq i\leq s}c_{i}=0$. Therefore, $q\left( x_{1}\right) \in \mathcal{I%
}$ implies 
\begin{equation*}
\mu :q\left( x_{1}\right) =\sum_{1\leq i\leq s}c_{i}\left(
x_{1}^{i}-1\right) \rightarrow \sum_{1\leq i\leq s}c_{i}i_{0}x_{2}^{i_{1}}%
\text{.}
\end{equation*}%
Next, 
\begin{eqnarray*}
\mu &:&x_{1}^{u}q\left( x_{1}\right) =\sum_{1\leq i\leq s}c_{i}\left(
x_{1}^{u}\left( x_{1}^{i}-1\right) \right) \rightarrow \sum_{1\leq i\leq
s}c_{i}\left( \left( u_{0}+i_{0}\right) x_{2}^{i_{1}}-u_{0}\right)
x_{2}^{u_{1}} \\
&=&\left( \sum_{1\leq i\leq
s}c_{i}u_{0}x_{2}^{i_{1}}+c_{i}i_{0}x_{2}^{i_{1}}-c_{i}u_{0}\right)
x_{2}^{u_{1}} \\
&=&\left( u_{0}\sum_{1\leq i\leq s}c_{i}\left( x_{2}^{i_{1}}-1\right)
+\sum_{1\leq i\leq s}c_{i}i_{0}x_{2}^{i_{1}}\right) x_{2}^{u_{1}} \\
&=&\left( u_{0}\sum_{1\leq i\leq s}c_{i}\left( x_{2}^{i_{1}}-1\right)
+q\left( x_{1}\right) ^{\mu }\right) x_{2}^{u_{1}}\text{;}
\end{eqnarray*}%
therefore,%
\begin{equation*}
\left( x_{1}^{u}q\left( x_{1}\right) \right) ^{\mu }-x_{2}^{u_{1}}.q\left(
x_{1}\right) ^{\mu }=u_{0}\left( \sum_{1\leq i\leq s}c_{i}\left(
x_{2}^{i_{1}}-1\right) \right) x_{2}^{u_{1}}\text{.}
\end{equation*}
\end{proof}

\begin{lemma}
Let $\mathcal{K}$ be an invariant $\mu $-ideal. If $q\left( x_{j}\right) \in 
\mathcal{K}$ for some $j$ then $q\left( x_{j}\right) =0$.
\end{lemma}

\begin{proof}
Given $q\left( x\right) =c_{0}+c_{1}x+..+c_{s}x^{s}\not=0$, define $e\left(
q\right) =\left\{ i\mid c_{i}\not=0\right\} $ and $\lambda \left( q\left(
x\right) \right) =\sum \left\{ i\mid i\in e\left( q\right) \right\} $. On
writing $0\leq $ $i=i_{0}+i_{1}p$ $\left( 0\leq i_{0}\leq p-1\right) $, 
\begin{equation*}
\lambda \left( q\left( x\right) \right) =\sum \left\{ i_{0}\mid i\in e\left(
q\right) \right\} +\left( \sum \left\{ i_{1}\mid i\in e\left( q\right)
\right\} \right) p.
\end{equation*}

Suppose there exists a nonzero polynomial $q\left( x\right) $ such that $%
q\left( x_{j}\right) \in $ $\mathcal{K}$ for some $j$ and choose one with
minimum $\lambda \left( q\left( x\right) \right) $. We may assume $%
c_{0}\not=0$.

(\textbf{1}) Suppose $j=1$. $q\left( x_{1}\right) \in $ $\mathcal{K}$ . Then
on choosing $u_{0}\not=0$, by the previous lemma, 
\begin{equation*}
\left( x_{1}^{u}q\left( x_{1}\right) \right) ^{\mu }-x_{2}^{u_{1}}.q\left(
x_{1}\right) ^{\mu }=u_{0}\left( \sum_{1\leq i\leq s}c_{i}\left(
x_{2}^{i_{1}}-1\right) \right) x_{2}^{u_{1}}\in \mathcal{K}
\end{equation*}%
and therefore, $l\left( x_{2}\right) =\sum_{1\leq i\leq s}c_{i}\left(
x_{2}^{i_{1}}-1\right) \in \mathcal{K}$ and $\lambda \left( l\left( x\right)
\right) \leq \sum \left\{ i_{1}\mid i\in e\left( q\right) \right\} $. Thus,
either $\lambda \left( q\left( x\right) \right) =\lambda \left( l\left(
x\right) \right) $ or $l\left( x_{2}\right) =0$. In the first case,%
\begin{equation*}
\sum \left\{ i_{0}\mid i\in e\left( q\right) \right\} +\left( \sum \left\{
i_{1}\mid i\in e\left( q\right) \right\} \right) p\leq \sum \left\{
i_{1}\mid i\in e\left( q\right) \right\} ,
\end{equation*}%
\begin{eqnarray*}
&& \\
\sum \left\{ i_{0}\mid i\in e\left( q\right) \right\} &\leq &\left( \sum
\left\{ i_{1}\mid i\in e\left( q\right) \right\} \right) \left( 1-q\right) ,
\\
\sum \left\{ i_{0}\mid i\in e\left( q\right) \right\} \text{ } &=&\sum
\left\{ i_{1}\mid i\in e\left( q\right) \right\} =0, \\
q\left( x_{1}\right) &=&c_{0}\in \mathcal{K}\text{.}
\end{eqnarray*}
Thus, $c_{0}=0$; contradiction.

In the second case, 
\begin{eqnarray*}
l\left( x_{2}\right) &=&\sum_{1\leq i\leq s}c_{i}\left(
x_{2}^{i_{1}}-1\right) =0, \\
l\left( x_{2}\right) &=&l\left( x_{1}\right) ^{\alpha }=0.
\end{eqnarray*}%
and as $\alpha $ is a monomorphism, it follows that $l\left( x_{1}\right) =0$%
.

Define $L_{j}=\left\{ i\in e\left( q\right) \mid i_{1}=j\right\} $ and let $%
t $ be such that $p^{t}\leq s<p^{t+1}$. Then,

\begin{eqnarray*}
l\left( x_{1}\right) &=&\sum_{1\leq i\leq s}c_{i}\left(
x_{1}^{i_{1}}-1\right) =\sum_{i\in L_{0}}c_{i}\left( x_{1}^{i_{1}}-1\right)
+\sum_{i\in L_{1}}c_{i}\left( x_{1}^{i_{1}}-1\right) +...+\sum_{i\in
L_{t}}c_{i}\left( x_{1}^{i_{1}}-1\right) \\
&=&\sum_{i\in L_{1}}c_{i}\left( x_{1}-1\right) +\sum_{i\in L_{2}}c_{i}\left(
x_{1}^{2}-1\right) +...+\sum_{i\in L_{t}}c_{i}\left( x_{1}^{t}-1\right) \\
&=&\left( \sum_{1\leq i\leq p-1}c_{i}\right) \left( x_{1}-1\right) +\left(
\sum_{0\leq i\leq p-1}c_{p+i}\right) \left( x_{1}^{2}-1\right) +...+\left(
\sum_{0\leq i\leq p-1}c_{p^{t}+i}\right) \left( x_{1}^{s}-1\right) =0
\end{eqnarray*}%
and%
\begin{eqnarray*}
\left( \sum_{1\leq i\leq p-1}c_{i}\right) &=&\left( \sum_{0\leq i\leq
p-1}c_{p+i}\right) =...=\left( \sum_{0\leq i\leq p-1}c_{p^{t}+i}\right) =0,
\\
\sum_{1\leq i\leq s}c_{i} &=&0\text{.}
\end{eqnarray*}%
Since $\sum_{0\leq i\leq s}c_{i}=0$, we reach $c_{0}=0$; a contradiction.

(\textbf{2}) Suppose $q\left( x_{j}\right) \in \mathcal{K}$ for some $2\leq
j\leq d$. Then, $q\left( x_{j}\right) \left( x_{1}^{u}-1\right) \in \mathcal{%
K}$ and $\mu :q\left( x_{j}\right) \left( x_{1}^{u}-1\right) \rightarrow
q\left( x_{j+1}\right) u_{0}x_{2}^{u_{1}}$; therefore $q\left(
x_{j+1}\right) \in \mathcal{K}$ which leads us back to $q\left( x_{1}\right)
\in \mathcal{K}.$
\end{proof}

\begin{lemma}
$\mathcal{K}=\left\{ 0\right\} $.
\end{lemma}

\begin{proof}
Suppose $\mathcal{K}\not=\left\{ 0\right\} $. Choose a polynomial $w\not=0$
in $\mathcal{K}$ having a minimum number of variables. Using the argument in
(\textbf{2}) above, we may assume one of the variables to be $x_{1}$. Let $%
\delta _{i}\left( w\right) $ be the $x_{i}$th degree of $w$ and 
\begin{equation*}
\delta \left( w\right) =\sum_{1\leq i\leq d}\delta _{i}\left( w\right) 
\end{equation*}%
the total degree of $w$. Then $\delta \left( w\right) \not=0$. Choose $w$
having minimum $\delta \left( w\right) $.

(\textbf{1}) Write $w=w\left( x_{1},x_{j_{1}},...,x_{j_{t}}\right)
=\sum_{i=0,...,s}w_{i}\left( x_{j_{1}},...,x_{j_{t}}\right) x_{1}^{i}$ where 
$2\leq j_{1}<...<j_{t}\leq d$. Then, 
\begin{equation*}
\delta \left( w\right) =\sum_{j\not=1}\delta _{j}\left( w\right) +s\text{.}
\end{equation*}%
We note that in 
\begin{equation*}
w=\sum_{i=0,...,s}w_{i}\left( x_{j_{1}},...,x_{j_{t}}\right)
+\sum_{i=1,...,s}w_{i}\left( x_{j_{1}},...,x_{j_{t}}\right) \left(
x_{1}^{i}-1\right)
\end{equation*}%
$w\in \mathcal{K}$ and second term on the right hand side is in $\mathcal{I}$%
; it follows that the first term is in $\mathcal{B}^{\prime }$.

Thus,%
\begin{eqnarray*}
\mu &:&w\rightarrow 0+\sum_{i=1,...,s}w_{i}\left(
x_{j_{1}+1},...,x_{j_{t}+1}\right) \left( i_{0}x_{2}^{i_{1}}\right) \\
&=&\sum_{i=1,...,s}w_{i}\left( x_{j_{1}+1},...,x_{j_{t}+1}\right) ix_{2}^{ 
\left[ \frac{i}{p}\right] }
\end{eqnarray*}%
and 
\begin{equation*}
\delta \left( w^{\mu }\right) \leq \sum_{j\not=1}\delta _{j}\left( w\right)
+ \left[ \frac{s}{p}\right] \leq \delta \left( w\right)
=\sum_{j\not=1}\delta _{j}\left( w\right) +s.
\end{equation*}%
Since $w^{\mu }\in \mathcal{K}$, by the minimality of $\delta \left(
w\right) $, we conclude 
\begin{equation*}
w^{\mu }=0\text{ or }1\leq s\leq p-1\text{.}
\end{equation*}

(\textbf{1.1}) Suppose $s=1$. Then, $w=w_{0}\left(
x_{j_{1}},...,x_{j_{t}}\right) +w_{1}\left( x_{j_{1}},...,x_{j_{t}}\right)
x_{1}$ and $w_{0}\not=0\not=w_{1}$. Since $w^{\mu }=w_{1}\left(
x_{j_{1}+1},...,x_{j_{t}+1}\right) \in \mathcal{K}$, by the minimality of $w$%
, we deduce $w_{1}\left( x_{j_{1}+1},...,x_{j_{t}+1}\right) =0$. However, as 
$w^{\mu }=w_{1}\left( x_{j_{1}+1},...,x_{j_{t}+1}\right) =w_{1}\left(
x_{j_{1}},...,x_{j_{t}}\right) ^{\alpha }$ and $\alpha $ monomorphism, we
have $w_{1}\left( x_{j_{1}},...,x_{j_{t}}\right) =0$; contradiction.

(\textbf{2}) Define $W_{j}=w.\left( x_{1}^{j}-1\right) $ for $1\leq j\leq
p-1 $. We apply the above analysis to $W_{j}$.

Here \ we have 
\begin{eqnarray*}
W_{j} &=&w.\left( x_{1}^{j}-1\right) =\left( \sum_{i=0,...,s}w_{i}\left(
x_{j_{1}},...,x_{j_{t}}\right) +\sum_{i=1,...,s}w_{i}\left(
x_{j_{1}},...,x_{j_{t}}\right) \left( x_{1}^{i}-1\right) \right) \left(
x_{1}^{j}-1\right) \\
&=&\sum_{i=0,...,s}w_{i}\left( x_{j_{1}},...,x_{j_{t}}\right) \left(
x_{1}^{j}-1\right) +\sum_{i=1,...,s}w_{i}\left(
x_{j_{1}},...,x_{j_{t}}\right) \left( x_{1}^{i}-1\right) \left(
x_{1}^{j}-1\right) \\
&=&\sum_{i=0,...,s}w_{i}\left( x_{j_{1}},...,x_{j_{t}}\right) \left(
x_{1}^{j}-1\right) \\
&&+\sum_{i=1,...,s}w_{i}\left( x_{j_{1}},...,x_{j_{t}}\right) \left( \left(
x_{1}^{i+j}-1\right) -\left( x_{1}^{i}-1\right) -\left( x_{1}^{j}-1\right)
\right) \text{.}
\end{eqnarray*}%
Since $\left[ \frac{j}{p}\right] =0$, we have 
\begin{eqnarray*}
\mu &:&W_{j}\rightarrow \sum_{i=0,...,s}w_{i}\left(
x_{j_{1}+1},...,x_{j_{t}+1}\right) j_{0}x_{2}^{\left[ \frac{j}{p}\right] } \\
&&+\sum_{i=1,...,s}w_{i}\left( x_{j_{1}+1},...,x_{j_{t}+1}\right) \left(
\left( i+j\right) _{0}x_{2}^{\left[ \frac{i+j}{p}\right] }-i_{0}x_{2}^{\left[
\frac{i}{p}\right] }-j_{0}x_{2}^{\left[ \frac{j}{p}\right] }\right) \\
&=&\sum_{i=0,...,s}w_{i}\left( x_{j_{1}+1},...,x_{j_{t}+1}\right)
-\sum_{i=1,...,s}w_{i}\left( x_{j_{1}+1},...,x_{j_{t}+1}\right) \\
&&+\sum_{i=1,...,s}w_{i}\left( x_{j_{1}+1},...,x_{j_{t}+1}\right) \left(
i+j\right) _{0}x_{2}^{\left[ \frac{i+j}{p}\right] }-\sum_{i=1,...,s}w_{i}%
\left( x_{j_{1}+1},...,x_{j_{t}+1}\right) i_{0}x_{2}^{\left[ \frac{i}{p}%
\right] } \\
&=&w_{0}\left( x_{j_{1}+1},...,x_{j_{t}+1}\right)
j+\sum_{i=1,...,s}w_{i}\left( x_{j_{1}+1},...,x_{j_{t}+1}\right) \left(
i+j\right) x_{2}^{\left[ \frac{i+j}{p}\right] } \\
&&-\sum_{i=1,...,s}w_{i}\left( x_{j_{1}+1},...,x_{j_{t}+1}\right) ix_{2}^{ 
\left[ \frac{i}{p}\right] }\text{.}
\end{eqnarray*}%
As $w^{\mu }=\sum_{i=1,...,s}w_{i}\left( x_{j_{1}+1},...,x_{j_{t}+1}\right)
ix_{2}^{\left[ \frac{i}{p}\right] }\in \mathcal{K}$, we conclude that for $%
1\leq j\leq p-1$, 
\begin{equation*}
V_{j}=w_{0}\left( x_{j_{1}+1},...,x_{j_{t}+1}\right)
j+\sum_{i=1,...,s}w_{i}\left( x_{j_{1}+1},...,x_{j_{t}+1}\right) \left(
i+j\right) x_{2}^{\left[ \frac{i+j}{p}\right] }\in \mathcal{K}\text{.}
\end{equation*}%
For each $j$, either $V_{j}=0$ or $\delta \left( V_{j}\right) \geq 1$.

(\textbf{2.1}) Suppose for some $j$, $V_{j}\not=0$. Then 
\begin{eqnarray*}
\delta \left( w\right) &=&\sum_{j\not=1}\delta _{j}\left( w\right) +s\leq
\delta \left( V_{j}\right) \leq \sum_{j\not=1}\delta _{j}\left( w\right) + 
\left[ \frac{s+j}{p}\right] \text{,} \\
s &\leq &\left[ \frac{s+j}{p}\right] .
\end{eqnarray*}%
On writing $s=s_{0}+s_{1}p$ where $0\leq s_{0}\leq p-1$, we conclude 
\begin{equation*}
s_{0}+s_{1}p\leq \left[ \frac{s_{0}+s_{1}p+j}{p}\right] \leq \left[ \frac{%
s_{0}-1+\left( s_{1}+1\right) p}{p}\right] ;
\end{equation*}%
\begin{equation*}
\text{if }s_{0}=0\text{ then }s_{1}p\leq s_{1}\text{ and }s_{1}=0\text{
(impossible),}
\end{equation*}%
\begin{eqnarray*}
\text{if }s_{0} &\not=&0\text{ then }s_{0}+s_{1}p\leq s_{1}+1, \\
s_{1}\left( p-1\right) &\leq &1-s_{0}\leq 0\text{ and }s_{1}=0,s_{0}=1\text{.%
}
\end{eqnarray*}%
Hence, $s=1$; by (1.1), we have a contradiction.

(\textbf{2.2}) Suppose 
\begin{equation*}
V_{j}=w_{0}\left( x_{j_{1}+1},...,x_{j_{t}+1}\right)
j+\sum_{i=1,...,s}w_{i}\left( x_{j_{1}+1},...,x_{j_{t}+1}\right) \left(
i+j\right) x_{2}^{\left[ \frac{i+j}{p}\right] }=0
\end{equation*}%
for all $j$. Then 
\begin{equation*}
\left( V_{j}\right) ^{\alpha ^{-1}}=w_{0}\left(
x_{j_{1}},...,x_{j_{t}}\right) j+\sum_{i=1,...,s}w_{i}\left(
x_{j_{1}},...,x_{j_{t}}\right) \left( i+j\right) x_{1}^{\left[ \frac{i+j}{p}%
\right] p}=0
\end{equation*}%
for all $j$.

In particular, 
\begin{eqnarray*}
\left( V_{p-1}\right) ^{\alpha ^{-1}} &=&w_{0}\left(
x_{j_{1}},...,x_{j_{t}}\right) \left( p-1\right)
+\sum_{i=1,...,s}w_{i}\left( x_{j_{1}},...,x_{j_{t}}\right) \left(
i+p-1\right) x_{1}^{\left[ \frac{i+p-1}{p}\right] p} \\
&=&-w_{0}\left( x_{j_{1}},...,x_{j_{t}}\right) \\
&&+\left( 
\begin{array}{c}
w_{2}\left( x_{j_{1}},...,x_{j_{t}}\right) +2w_{3}\left(
x_{j_{1}},...,x_{j_{t}}\right) +... \\ 
+\left( p-2\right) w_{p-1}\left( x_{j_{1}},...,x_{j_{t}}\right)%
\end{array}%
\right) x_{1}^{p}+...=0
\end{eqnarray*}%
and therefore, $w_{0}\left( x_{j_{1}},...,x_{j_{t}}\right) =0$; a
contradiction.
\end{proof}

\subsubsection{Representation on the $p^{2}$-tree}

We index the $p^{2}$-tree by sequences from the set of strings $ij$ where $%
0\leq i,j\leq p-1$ and choose the following transversal for $H$: 
\begin{equation*}
T=\left\{ x_{1}^{i}a^{j}\mid 0\leq i,j\leq p-1\right\} \text{;}
\end{equation*}%
we will indicate $x_{1}^{i}a^{j}$ by $ij$.

\textbf{(1) Permutation representation. }Given $g\in G$, we denote the
permutation induced by $g$ on the transversal $T$ by $g^{\sigma }$. Then, 
\begin{eqnarray*}
a^{\sigma } &:&x_{1}^{i}a^{j}\rightarrow x_{1}^{i}a^{j+1}\text{ if }0\leq
j\leq p-2, \\
x_{1}^{i}a^{p-1} &\rightarrow &x_{1}^{i}, \\
a^{\sigma } &:&ij\rightarrow i\left( j+1\right) \text{ if }0\leq j\leq p-2%
\text{, } \\
i\left( p-1\right) &\rightarrow &i0;
\end{eqnarray*}%
\begin{eqnarray*}
\left( x_{1}\right) ^{\sigma } &:&x_{1}^{i}a^{j}\rightarrow x_{1}^{i+1}a^{j}%
\text{ if }0\leq i\leq p-2, \\
x_{1}^{p-1}a^{j} &\rightarrow &a^{j}, \\
ij &\rightarrow &\left( i+1\right) j\text{ if }0\leq i\leq p-2\text{, } \\
\left( p-1\right) j &\rightarrow &0j;
\end{eqnarray*}%
\begin{eqnarray*}
\left( x_{l}\right) ^{\sigma } &:&x_{1}^{i}a^{j}\rightarrow x_{1}^{i}a^{j}%
\text{ for }0\leq i\leq p-1,2\leq l\leq d, \\
ij &\rightarrow &ij\text{.}
\end{eqnarray*}

\textbf{(2)} \textbf{Cofactors (or Schreir elements) of actions on }$T$.

We calculate%
\begin{eqnarray*}
\text{ cofactors of }a &\text{:}&\text{ } \\
\text{ }\left( x_{1}^{i}a^{j}.a\right) \left( x_{1}^{i}a^{j+1}\right) ^{-1}
&=&1\text{ if }0\leq j\leq p-2, \\
\left( x_{1}^{i}a^{p-1}.a\right) \left( x_{1}^{i}\right) ^{-1} &=&1\text{;}
\end{eqnarray*}%
\begin{eqnarray*}
\text{ cofactors of }x_{1} &\text{:}& \\
\left( x_{1}^{i}a^{j}.x_{1}\right) \left( x_{1}^{i+1}a^{j}\right) ^{-1}
&=&x_{1}^{i}a^{j}x_{1}a^{-j}x_{1}^{-i-1}=\left( a^{j}\right)
^{x_{1}^{-i}}\left( a^{-j}\right) ^{x_{1}^{-\left( i+1\right) }}=\left(
a^{j}\right) ^{\left( x_{1}^{-i}-x_{1}^{-\left( i+1\right) }\right) } \\
&=&\left( a^{j}\right) ^{x_{1}^{-i}\left( 1-x_{1}^{-1}\right) }\text{ if }%
0\leq i\leq p-2, \\
\left( x_{1}^{p-1}a^{j}.x_{1}\right) a^{-j}
&=&x_{1}^{p-1}a^{j}x_{1}a^{-j}=x_{1}^{p}\left( a^{j}\right) ^{x-1}=\left(
a^{j}\right) ^{\left( x-1\right) x_{1}^{-p}}x_{1}^{p}\text{;}
\end{eqnarray*}%
\begin{eqnarray*}
\text{ cofactors of }x_{l}\text{ (}2 &\leq &l\leq d\text{) : } \\
\left( x_{1}^{i}a^{j}.x_{l}\right) \left( x_{1}^{i}a^{j}\right) ^{-1}
&=&x_{1}^{i}a^{j}x_{l}a^{-j}x_{1}^{-i}=\left( a^{j}\right)
^{x_{1}^{-i}\left( x_{l}-1\right) x_{l}^{-1}}x_{l}\text{.}
\end{eqnarray*}%
\textbf{(3)} \textbf{Images of cofactors under} $f$.

We calculate%
\begin{equation*}
f:\left( a^{j}\right) ^{x_{1}^{-i}\left( 1-x_{1}^{-1}\right) }\rightarrow
\left( a^{j}\right) ^{x_{2}^{-1}},
\end{equation*}%
\begin{equation*}
f:\left( a^{j}\right) ^{\left( x-1\right) x_{1}^{-p}}x_{1}^{p}\rightarrow
\left( a^{j}\right) ^{x_{2}^{-1}}x_{2};
\end{equation*}%
\begin{equation*}
f:\left( a^{j}\right) ^{x_{1}^{-i}\left( x_{l}-1\right)
x_{l}^{-1}}x_{l}\rightarrow \left( a^{j}\right)
^{-ix_{2}^{-1}x_{l+1}^{-1}\left( x_{l+1}-1\right) }x_{l+1}
\end{equation*}%
Therefore, 
\begin{eqnarray*}
\left( x_{1}^{\varphi }\right) _{ij} &=&\left( \left( a^{j}\right)
^{x_{2}^{-1}}\right) ^{\varphi }\text{ for }0\leq i\leq p-2 \\
&=&\left( \left( a^{j}\right) ^{x_{2}^{-1}}x_{2}\right) ^{\varphi }~\text{%
for }i=p-1
\end{eqnarray*}%
\begin{equation*}
\left( x_{l}^{\varphi }\right) _{ij}=\left( \left( a^{j}\right)
^{-ix_{2}^{-1}x_{l+1}^{-1}\left( x_{l+1}-1\right) }x_{l+1}\right) ^{\varphi }%
\text{ for }0\leq i\leq p-1\text{.}
\end{equation*}%
On identifying $a^{\varphi }$ with $a$, $x_{1}^{\varphi }$ with $x_{1}$ and $%
x_{l}^{\varphi }$ with $x_{l}$ we obtain%
\begin{eqnarray*}
\left( x_{1}\right) _{ij} &=&\left( a^{j}\right) ^{x_{2}^{-1}}\text{ for }%
0\leq i\leq p-2 \\
&=&\left( a^{j}\right) ^{x_{2}^{-1}}x_{2}~\text{for }i=p-1,
\end{eqnarray*}%
\begin{equation*}
\left( x_{l}\right) _{ij}=\left( a^{-ij}\right)
^{x_{2}^{-1}x_{l+1}^{-1}\left( x_{l+1}-1\right) }x_{l+1}\text{ for }0\leq
i\leq p-1\text{.}
\end{equation*}

\subsection{Proof of Theorem 5}

Let $p=2$ and re-index the $4$-tree by sequences from $\left\{
0,1,2,3\right\} $. Then the above representation becomes 
\begin{eqnarray*}
a &=&\left( 0,1\right) \left( 2,3\right) , \\
x_{1} &=&\left( 1,a^{x_{2}^{-1}},x_{2},a^{x_{2}^{-1}}x_{2}\right) \left(
0,2\right) \left( 1,3\right) , \\
x_{2} &=&\left(
x_{1},x_{1},x_{1},a^{x_{2}^{-1}+x_{1}^{-1}x_{2}^{-1}}x_{1}\right) =\left(
1,1,1,a^{x_{2}^{-1}+x_{1}^{-1}x_{2}^{-1}}\right) \left( x_{1}\right)
^{\left( 1\right) }\text{.}
\end{eqnarray*}

\textbf{(1)} \textbf{Powers of }$x_{1},x_{2}$:

\begin{eqnarray*}
x_{1}^{2n} &=&\left(
x_{2}^{n},x_{2}^{n},x_{2}^{n},a^{x_{2}^{-1}+x_{2}^{-\left( n+1\right)
}}x_{2}^{n}\right) , \\
x_{1}^{2n+1} &=&\left( x_{2}^{n},a^{x_{2}^{-\left( n+1\right)
}}x_{2}^{n},x_{2}^{n+1},a^{x_{2}^{-1}}x_{2}^{n+1}\right) \left( 0,2\right)
\left( 1,3\right) ,
\end{eqnarray*}%
\begin{equation*}
x_{2}^{n}=\left(
x_{1}^{n},x_{1}^{n},x_{1}^{n},a^{x_{2}^{-1}+x_{1}^{-n}x_{2}^{-1}}x_{1}^{n}%
\right)
\end{equation*}%
for all $n.$

\textbf{\ } \textbf{(2)} \textbf{Conjugates of }$a$:%
\begin{eqnarray*}
a^{x_{1}^{2n}} &=&\left(
1,1,a^{x_{2}^{n-1}+x_{2}^{-1}},a^{x_{2}^{n-1}+x_{2}^{-1}}\right) \left(
0,1\right) \left( 2,3\right) , \\
a^{x_{1}^{2n+1}} &=&\left(
a^{x_{2}^{n}},a^{x_{2}^{n}},a^{x_{2}^{-1}},a^{x_{2}^{-1}}\right) \left(
0,1\right) \left( 2,3\right) ,
\end{eqnarray*}%
\begin{equation*}
a^{x_{2}^{n}}=\left(
1,1,a^{x_{1}^{n}x_{2}^{-1}+x_{2}^{-1}},a^{x_{1}^{n}x_{2}^{-1}+x_{2}^{-1}}%
\right) \left( 0,1\right) \left( 2,3\right)
\end{equation*}%
\begin{eqnarray*}
a^{x_{1}^{2n}x_{2}^{l}} &=&\left(
1,1,a^{x_{1}^{l}x_{2}^{n-1}+x_{2}^{-1}},a^{x_{1}^{l}x_{2}^{n-1}+x_{2}^{-1}}%
\right) \left( 0,1\right) \left( 2,3\right) , \\
a^{x_{1}^{2n+1}x_{2}^{l}} &=&\left(
a^{x_{1}^{l}x_{2}^{n}},a^{x_{1}^{l}x_{2}^{n}},a^{x_{2}^{-1}},a^{x_{2}^{-1}}%
\right) \left( 0,1\right) \left( 2,3\right)
\end{eqnarray*}%
for all $n,l.$

\textbf{(3)} \textbf{Products of Conjugates of} $a$

Let $0<n_{1}<n_{2}<...<n_{2s}$. Then%
\begin{eqnarray*}
a^{x_{1}^{2n_{1}}+x_{1}^{2n_{2}}+...+x_{1}^{2n_{2s}}} &=&\left(
1,1,a^{\left( x_{2}^{n_{1}-1}+x_{2}^{n_{2}-1}...+x_{2}^{n_{s}-1}\right)
},a^{\left( x_{2}^{n_{1}-1}+x_{2}^{n_{2}-1}...+x_{2}^{n_{s}-1}\right)
}\right) , \\
a^{x_{1}^{2n_{1}}+x_{1}^{2n_{2}}+...+x_{1}^{2n_{2s+1}}} &=&\left(
1,1,a^{\left(
x_{2}^{n_{1}-1}+x_{2}^{n_{2}-1}+...+x_{2}^{n_{s+1}-1}+x_{2}^{-1}\right)
},a^{\left(
x_{2}^{n_{1}-1}+x_{2}^{n_{2}-1}+...+x_{2}^{n_{s+1}-1}+x_{2}^{-1}\right)
}\right) \\
&&.\left( 0,1\right) \left( 2,3\right) , \\
a^{1+x_{1}^{2n_{1}}+x_{1}^{2n_{2}}+...+x_{1}^{2n_{2s+1}}} &=&\left(
1,1,a^{\left(
x_{2}^{-1}+x_{2}^{n_{1}-1}+...+x_{2}^{n_{s}-1}+x_{2}^{n_{s+1}-1}\right)
},a^{\left(
x_{2}^{-1}+x_{2}^{n_{1}-1}+...+x_{2}^{n_{s}-1}+x_{2}^{n_{s+1}-1}\right)
}\right) \text{.}
\end{eqnarray*}

\textbf{(4)} \textbf{States of the automaton }$x_{1}=\left(
1,a^{x_{2}^{-1}},x_{2},a^{x_{2}^{-1}}x_{2}\right) \left( 0,2\right) \left(
1,3\right) $

Using the above equations together with 
\begin{eqnarray*}
a^{x_{1}^{2n}x_{2}^{m}}x_{1} &=&\left(
a^{x_{2}^{-1}},1,a^{x_{1}^{m}x_{2}^{n-1}}x_{2},a^{x_{1}^{m}x_{2}^{n-1}+x_{2}^{-1}}x_{2}\right) \left( 0,3\right) \left( 1,2\right) ,
\\
a^{x_{1}^{2n+1}x_{2}^{m}}x_{1} &=&\left(
a^{x_{1}^{m}x_{2}^{n}}a^{x_{2}^{-1}},a^{x_{1}^{m}x_{2}^{n}},x_{2},a^{x_{2}^{-1}}x_{2}\right) \left( 0,3\right) \left( 1,2\right) 
\text{,}
\end{eqnarray*}%
we calculate the states of $x_{1}$: 
\begin{eqnarray*}
a^{x_{2}^{-1}} &=&\left(
1,1,a^{x_{1}^{-1}x_{2}^{-1}+x_{2}^{-1}},a^{x_{1}^{-1}x_{2}^{-1}+x_{2}^{-1}}%
\right) \left( 0,1\right) \left( 2,3\right) \\
a^{x_{2}^{-1}}x_{2} &=&\left(
x_{1},x_{1},x_{1},a^{x_{1}^{-1}x_{2}^{-1}+x_{2}^{-1}}x_{1}\right) \left(
0,1\right) \left( 2,3\right) , \\
a^{x_{2}^{-1}+x_{1}^{-1}x_{2}^{-1}} &=&\left(
a^{x_{1}^{-1}x_{2}^{-1}},a^{x_{1}^{-1}x_{2}^{-1}},a^{x_{1}^{-1}x_{2}^{-1}},a^{x_{1}^{-1}x_{2}^{-1}}\right)
\\
a^{x_{2}^{-1}+x_{1}^{-1}x_{2}^{-1}}x_{1} &=&\left(
a^{x_{1}^{-1}x_{2}^{-1}},a^{x_{1}^{-1}x_{2}^{-1}+x_{2}^{-1}},a^{x_{1}^{-1}x_{2}^{-1}}x_{2},a^{x_{1}^{-1}x_{2}^{-1}+x_{2}^{-1}}x_{2}\right) \left( 0,2\right) \left( 1,3\right)
\end{eqnarray*}%
\begin{eqnarray*}
a^{x_{1}^{-1}x_{2}^{-1}} &=&\left(
a^{x_{1}^{-1}x_{2}^{-1}},a^{x_{1}^{-1}x_{2}^{-1}},a^{x_{2}^{-1}},a^{x_{2}^{-1}}\right) \left( 0,1\right) \left( 2,3\right)
\\
a^{x_{1}^{-1}x_{2}^{-1}}x_{2} &=&\left(
a^{x_{1}^{-1}x_{2}^{-1}}x_{1},a^{x_{1}^{-1}x_{2}^{-1}}x_{1},a^{x_{1}^{-1}x_{2}^{-1}}x_{1},a^{x_{2}^{-1}}x_{1}\right) \left( 0,1\right) \left( 2,3\right)
\\
a^{x_{2}^{-1}+x_{1}^{-1}x_{2}^{-1}}x_{2} &=&\left(
a^{x_{1}^{-1}x_{2}^{-1}}x_{1},a^{x_{1}^{-1}x_{2}^{-1}}x_{1},a^{x_{1}^{-1}x_{2}^{-1}}x_{1},a^{x_{2}^{-1}}x_{1}\right)
\end{eqnarray*}%
\begin{eqnarray*}
a^{x_{1}^{-1}x_{2}^{-1}}x_{1} &=&\left(
a^{x_{1}^{-1}x_{2}^{-1}+x_{2}^{-1}},a^{x_{1}^{-1}x_{2}^{-1}},x_{2},a^{x_{2}^{-1}}x_{2}\right) \left( 0,3\right) \left( 1,2\right)
\\
a^{x_{2}^{-1}}x_{1} &=&\left(
a^{x_{2}^{-1}},1,a^{x_{1}^{-1}x_{2}^{-1}}x_{2},a^{x_{1}^{-1}x_{2}^{-1}+x_{2}^{-1}}x_{2}\right) \left( 0,3\right) \left( 1,2\right) 
\text{.}
\end{eqnarray*}%
Therefore $x_{1}$ has in total $12$ states: 
\begin{equation*}
\left\{ 
\begin{array}{c}
e,x_{1},x_{2}, \\ 
a^{x_{2}^{-1}},a^{x_{2}^{-1}}x_{1},a^{x_{2}^{-1}}x_{2}, \\ 
a^{x_{1}^{-1}x_{2}^{-1}},a^{x_{1}^{-1}x_{2}^{-1}}x_{1},a^{x_{1}^{-1}x_{2}^{-1}}x_{2},
\\ 
a^{x_{1}^{-1}x_{2}^{-1}+x_{2}^{-1}},a^{x_{2}^{-1}+x_{1}^{-1}x_{2}^{-1}}x_{1},a^{x_{1}^{-1}x_{2}^{-1}+x_{2}^{-1}}x_{2}%
\end{array}%
\right\} \text{.}
\end{equation*}%
Since $x_{2}$ is a state of $x_{1}$ it follows that $G$ is a finite state
group.

\textbf{(5)} \textbf{Incidence matrix for the graph of states of }$x_{1}$.

Rewrite the states of \textbf{\ }$x_{1}$ as%
\begin{eqnarray*}
s_{1} &=&x_{1},s_{2}=x_{2}, \\
s_{3} &=&a^{x_{2}^{-1}},s_{4}=a^{x_{2}^{-1}}x_{1},s_{5}=a^{x_{2}^{-1}}x_{2},
\\
s_{6}
&=&a^{x_{1}^{-1}x_{2}^{-1}},s_{7}=a^{x_{1}^{-1}x_{2}^{-1}}x_{1},s_{8}=a^{x_{1}^{-1}x_{2}^{-1}}x_{2},
\\
s_{9}
&=&a^{x_{1}^{-1}x_{2}^{-1}+x_{2}^{-1}},s_{10}=a^{x_{1}^{-1}x_{2}^{-1}+x_{2}^{-1}}x_{1},s_{11}=a^{x_{1}^{-1}x_{2}^{-1}+x_{2}^{-1}}x_{2}%
\text{. }
\end{eqnarray*}%
The incidence matrix of the automaton $x_{1}$ reads as follows:%
\begin{equation*}
\left( 
\begin{array}{ccccccccccccc}
& e & s_{1} & s_{2} & s_{3} & s_{4} & s_{5} & s_{6} & s_{7} & s_{8} & s_{9}
& s_{10} & s_{11} \\ 
e & 4 &  &  &  &  &  &  &  &  &  &  &  \\ 
s_{1} &  & 1 & 1 & 1 &  & 1 &  &  &  &  &  &  \\ 
s_{2} &  & 3 &  &  &  &  &  &  &  &  & 1 &  \\ 
s_{3} & 2 &  &  &  &  &  &  &  &  & 2 &  &  \\ 
s_{4} & 1 &  &  & 1 &  &  &  &  & 1 &  &  & 1 \\ 
s_{5} &  & 3 &  &  &  &  &  &  &  &  & 1 &  \\ 
s_{6} &  &  &  & 2 &  &  & 2 &  &  &  &  &  \\ 
s_{7} &  &  & 1 &  &  & 1 & 1 &  &  & 1 &  &  \\ 
s_{8} &  &  &  &  & 1 &  &  & 3 &  &  &  &  \\ 
s_{9} &  &  &  &  &  &  & 4 &  &  &  &  &  \\ 
s_{10} &  &  &  &  &  &  & 1 &  & 1 & 1 &  & 1 \\ 
s_{11} &  &  &  &  & 1 &  &  & 3 &  &  &  & 
\end{array}%
\right) \text{.}
\end{equation*}

.

.


\begin{thebibliography}{99}
\bibitem{KaimVersh} Kaimanovich, V., Vershik, A., Random walks on discrete
groups: Boundary and Entropy, The Annals of Probability, vol. 11, (1983),
457-490.

\bibitem{GrigZuk} Grigorchuk, R., Zuk, A., The Lamplighter group as a group
generated by a 2-state automaton and its spectrum. Geometriae Dedicata 87,
(2001), 209-244.

\bibitem{GLSZ} Grigorchuk, R., Linnel, P., Schick, Th., and Zuk, A., On a
conjecture of Atiyah. C.R. Acad. Sci. Paris 331, Serie I, (2000), 663-668.

\bibitem{SilvaStein} Silva, P., Steinberg, B., On a class of automata groups
generalizing Lamplighter groups, Int. J. Algebra Comut. 15 (2005), 1213-1234.

\bibitem{KSS} Kambides, M., Silva, P., Steinberg, B., The spectra of
Lamplighter groups and Cayley machines, Geom. Dedicata 120, (2006), 193-227.

\bibitem{BarthSunik} Bartholdi, L., Sunik, Z., Some solvable automata
groups, Contemp. Math. 394, (2006), 11-30.

\bibitem{BonDaRo} Bondarenko, I., D'Angeli, D., Rodaro, E., The Lamplighter
group $Z_{3}\wr Z$ generated by a bireversible automaton, arXiv:1502.07981

\bibitem{GLN} Grigorchuk, R., Leemann, P-H, Nagnibeda, T., Lamplighter
groups, de Bruijn graphs, spider-web graphs and their spectra,
arXiv:1502.06722.

\bibitem{Sidki} Sidki, S., Tree wreathing applied to the generation of
groups by finite automata, International Journal of Algebra and Computation,
vol. 15. (2005), 1-8.

\bibitem{NekSid} Nekrachevych,V., Sidki, S., Automorphisms of the binary
tree: state-closed subgroups and dynamics of 1/2 endomorphisms, In: Groups:
London Mathematical Lecture Notes Series, Vol. 311, (2004), 375-404.
Topological, Combinatorial and Arithmetic Aspects, Muller, T. W., (Ed.)

\bibitem{Nek} Nekrashevych, V., Self-similar groups, Math. Surveys and
Monographs, 117, American Mathematical Society, Providence, RI, 2005.

\bibitem{BruSid} Brunner, A., Sidki, S., Abelian state-closed subgroups of
automorphisms of m-ary trees, Groups, Geometry and Dynamics, vol.3, (2010)
455-472.

\bibitem{BerSid} Berlatto, A., Sidki, S., Virtual endomorphisms of nilpotent
groups, Groups, Geometry and Dynamics, vol.1, (2007), 21-46.

\bibitem{Kapo} Kapovich, M., Arithmetic aspects of self-similar groups,
Groups Geometry and Dynamics 6 (2012), 737--754.

\bibitem{MuntSav} Y. Muntyan, D. Savchuk, AutomGrp -- GAP package for
computations in self-similar groups and semigroups, Version 1.2.4, 2014.
\end{thebibliography}
\end{document}